\documentclass[oneside,english,reqno]{amsart}

\usepackage{xparse}
\let\realItem\item 
\makeatletter
\NewDocumentCommand\myItem{ o }{%
   \IfNoValueTF{#1}%
      {\realItem}
      {\realItem[#1]\def\@currentlabel{#1}}
}
\makeatother

\usepackage{enumitem}    
\setlist[enumerate]{
    before=\let\item\myItem,       
    label=\textnormal{(\arabic*)}, 
    widest=(2')                    
}

\usepackage{color}
\usepackage{amssymb}
\usepackage{amsmath}
\usepackage{amsthm}
\usepackage{bbm}
\usepackage{amscd}
\usepackage{mathrsfs}
\usepackage{tikz}
\usepackage{bm}
\theoremstyle{definition}
\newtheorem{dfn}{Definition}[section]
\theoremstyle{plain}
\newtheorem{thm}[dfn]{Theorem}

\newtheorem{lm}[dfn]{Lemma}

\theoremstyle{remark}

\newtheorem{rem}[dfn]{Remark}

\newcommand{\N}{\mathbb{N}}
\newcommand{\quotient}[2]{
\mathchoice{  \text{\raise1ex\hbox{$#1$}\!\Big/\!\lower1ex\hbox{$#2$}} }
                  {  \text{\raise1pt\hbox{$#1$}\big/\lower1pt\hbox{$#2$}} }
                  {  {#1}\,/\,{#2}  }
                  {  {#1}\,/\,{#2}  }
}

\numberwithin{equation}{section}

\usepackage[linktocpage,colorlinks,linkcolor=blue,anchorcolor=blue,citecolor=blue,urlcolor=black,pagebackref]{hyperref}

\begin{document}

\title{Double $q$-Wigner Chaos and the Fourth Moment}
\author{Todd Kemp}
\address{Department of Mathematics, University of California, San Diego 9500 Gilman Drive \#0112 La Jolla, CA  92093-0112, USA}
\email{tkemp@ucsd.edu}
\thanks{T. Kemp supported in part by NSF Grant DMS-2400246}
\author{Akihiro Miyagawa}
\address{Department of Mathematics, University of California, San Diego 9500 Gilman Drive \#0112 La Jolla, CA  92093-0112, USA}
\email{amiyagwa@ucsd.edu}
\thanks{A. Miyagawa acknowledges support from JSPS Overseas Research Fellowship.}

\begin{abstract} In this paper, we prove the Fourth Moment Theorem for sequences of (noncommutative) random variables given as sums of two stochastic integrals in two different parity orders of chaos, both in the free Wigner chaos setting and a $q$-Gaussian generalization.  Specifically, we prove that convergence to the appropriate central limit distribution is mediated entirely by the behavior of the first four (mixed) moments of the two stochastic integrals, which in turn controls the $L^2$ norms of partial integral contractions of those kernels.  The key step in both the free and $q$-Gaussian settings is a polarization identity for fourth cumulants of sums which holds only when the two terms have differing parities.  These results are analogous to the recent preprint {\em Fourth-Moment Theorems for Sums of Multiple Integrals} by Basse-O'Connor, Kramer-Bang, and Svedsen in the classical Wiener-It\^o chaos setting.
\end{abstract}

\maketitle


\section{Introduction}\label{Introduction}

In classical stochastic analysis, the Wiener chaos plays a foundational role.  Let $W=(W_t)_{t\ge 0}$ denote a standard Brownian motion on $\mathbb{R}$, defined on a probability space $(\Omega,\mathscr{F},\mathbb{P})$ (i.e.\ $\Omega$ may be taken as the path space over $\mathbb{R}$ with $\mathbb{P}$ the Wiener measure started at $0$).  For any positive integer $n$, and any function $f\in L^2(\mathbb{R}_+^n)$, the multiple Wiener--It\^o stochastic integral
\[ I^W(f) = \int f(t_1,\ldots,t_n)\,dW_{t_1}\cdots dW_{t_n} \]
is an $L^2(\Omega,\mathscr{F},\mathbb{P})$ random variable; indeed $\|I_n^W(f)\|_{L^2}^2 = n!\|f\|_2^2$. The image of $L^2(\mathbb{R}_+^n)$ under $I^W$ is the {\bf $n$th Wiener chaos} $\Gamma^n$; it is a closed subspace of $L^2(\Omega,\mathscr{F},\mathbb{P})$, and indeed $L^2(\Omega,\mathscr{F},\mathbb{P}) = \bigoplus_{n\ge 0} \Gamma^n$ decomposes orthogonally.  Precisely: every $L^2(\Omega,\mathscr{F},\mathbb{P})$ random variable $X$ has an orthogonal expansion $X = \sum_{n\ge 0} I^W(f_n)$ for some sequence $f_n\in L^2(\mathbb{R}_+^n)$ with $n!\|f_n\|_2^2 = 1$; and the sequence $f_n$ is unique if the functions are fully symmetric: $f_n(t_1,\ldots,t_n) = f_n(t_{\sigma 1},\ldots,t_{\sigma n})$ for any permutation $\sigma\in S_n$.

Properties of random variables in different orders of chaos have been studied extensively over the last 75 years or more.  It is known from early on, for example, that if $X\in \Gamma^n$ with $n\ge 1$ then the distribution of $X$ has a smooth density with tails comparable to those of the $n$th power of a normal random variable, \cite{Janson-Book}.  The chaos expansion (and multivaratiate / infinite dimensional versions developed by Cameron) played a key foundational role in the development of Malliavin calculus, cf. \cite{Nualart-Book}.

In 2005, Nualart and Peccati \cite{NuPec2005} considered the question of normal approximation within a fixed order of chaos.  They proved a landmark result, now known as a {\em Fourth Moment Theorem}, which follows.  (See Section \ref{sect.Wigner.chaos} below for a full description of the integral contraction operators $\overset{p}{\frown}$.)

\begin{thm}[Nourdin, Peccati, \cite{NuPec2005}] \label{thm.NuPec} Let $f_k$ be a sequence of real-valued symmetric functions in $L^2(\mathbb{R}^n_+)$ normalized so that $n!\|f_k\|_2\to 1$ as $k\to\infty$.  Then the following three conditions are equivalent.
\begin{enumerate}
    \item $\displaystyle{\lim_{k\to\infty}\mathbb{E}[I^W(f_k)^4] = 3}$.
    \item $\displaystyle{\lim_{k\to\infty} f_k \overset{p}{\frown} f_k = 0}$ in $L^2(\mathbb{R}_+^{2(n-p)})$ for $p=1,\ldots,n-1$.
    \item The sequence $(I^W(f_k))_{k\in\N}$ converges in distribution to a standard Gaussian random variable.
\end{enumerate}
\end{thm}
Of course, convergence in distribution to a standard Gaussian implies convergence of the fourth moment to $3$, the fourth moment of a standard Gaussian.  That the converge holds true in a fixed order chaos was a remarkable realization.  This result has been generalized substantially to include other equivalent conditions linking to Malliavin calculus and Stein's method, giving tight quantitative bounds on the rate of convergence, and also allowing similar results for Poisson and other distributional limit theorems; see \cite{NoPec-Book,Goldstein-Book} and the references therein.

Meanwhile, connecting with the development of noncommutative probability theories (in the 1980s and 1990s), the Wiener--It\^o chaos expansion can be realized in terms of Boson Fock spaces from quantum theory.  Given a Hilbert space $H$, the Boson Fock space $\mathcal{F}_+(H)$ is the completion of the symmetric tensor algebra
\[ \mathcal{F}_+(H) = \mathbb{C}\Omega\oplus\bigoplus_{n=1}^\infty H^{\odot n} \]
in the symmetrized tensor product norm from $H$ boosted to each $H^{\odot n}$ (which makes these different ``particle spaces'' orthogonal); here $\Omega$ is a unit length vector separate from $H$ called the ``vacuum'', and $\odot$ is the symmetric tensor product.  For each $h\in H$, the creation operator $c(h)$ is a densely-defined adjointable unbounded operator on $\mathcal{F}_+(H)$ given by $c(h)\psi = h\odot\psi$.  The selfadjoint ``field operators $\omega(h) = c(h)+c(h)^\ast$ have Gaussian distribution (with mean $0$ and variance $\|h\|_H^2$) in terms of the vacuum expectation state $\tau_+(X) = \langle X\Omega,\Omega\rangle_{\mathcal{F}_+(H)}$.  More generally, Hudson and Parthasarathy showed \cite{HP1984} that, taking $H=L^2(\mathbb{R}_+)$, the noncommutative stochastic process $W_t = \omega(\mathbbm{1}_{[0,t]})$ has the same finite-dimensional distributions as standard Brownian motion.  Moreover, using this Brownian motion to construct Wiener--It\^o integrals $I^W$ (therefore valued in the unbounded operators on $\mathcal{F}_+(H)$ with $H=L^2(\mathbb{R}_+)$), one finds the identity
\begin{equation}\label{eq.Parth.id} I^W(h_1\odot\cdots\odot h_n)\Omega = h_1\odot\cdots\odot h_n. \end{equation}
This sets up an isomorphism between the Wiener--It\^o chaos of order $n$ and the $n$-particle space in $\mathcal{F}_+(L^2(\mathbb{R}_+))$.  We take this perspective in much of this paper, working directly in Fock spaces.

The Boson Fock space was originally constructed as an arena to verify the canonical commutation relations (CCR) between the creation and annihilation operators: if $H$ is separable and $\{e_i\}_{i\in\N}$ is an orthonormal basis for $H$, the operators $c_i = c(e_i)$ satisfy $c_i^\ast c_j - c_jc_i^\ast = \delta_{ij}$, commutations relations from quantum field theory associated to the (quantum) statistical physics of Boson states.  Similarly, Fermions are characterized by the canonical {\em anti}commutation relations (CAR), $c_i^\ast c_j + c_jc_i^\ast = \delta_{ij}$, which can be realized on the {\em Fermion} Fock space $\mathcal{F}_-(H)$ (using skew-symmetric tensor products).  Interpolating between these are the {\bf canonical $q$-commutation relations} ($q$-CCR) $c_i^\ast c_j - q c_j c_i^\ast = \delta_{ij}$ for $-1<q<1$.  Introduced and studied (without being rigorously constructed) by Frisch and Bourret starting in 1970 \cite{MR260352}, they were put on a rigorous mathematical footing by Bo\.zejko and Speicher \cite{MR1105428}, constructed on $q$-deformed Fock spaces $\mathcal{F}_q(H)$.  The special case $q=0$ is the Boltzmann Fock space, the completion of the fully non-symmetric tensor algebra over $H$:
\[ \mathcal{F}_0(H) = \mathbb{C}\Omega\oplus\bigoplus_{n=1}^\infty H^{\otimes n}. \]
The $q$-Fock space is a ``twisting'' of the Boltzmann Fock space cf.\ Section \ref{sect.q-Wigner.chaos}.  For $|q|<1$, the creation and annihilation operators are in fact bounded, and irreducible (i.e.\ the von Neumann algebra generated by $\{c_q(h)\colon h\in H\}$ is all of $\mathscr{B}(\mathcal{F}_q(H))$, cf.\ \cite{Kemp2005}).  The von Neumann algebra generated by the field operators $\omega_q(e_i) = c_q(e_i)+c_q(e_i)^\ast$, denoted $\Gamma_q(H)$, is a tracial von Neumann factor, where the unique tracial state is the vacuum expectation $\tau_q(X) = \langle X\Omega,\Omega\rangle_{\mathcal{F}_q(H)}$.  These are the {\em $q$-Gaussian factors}.  The special case $q=0$ corresponds to the realm of free probability.

In general, for $|q|<1$, the map from $\Gamma_q(H)$ into $\mathcal{F}_q(H)$ defined by $X\mapsto X\Omega$ is an isometry (by definition of the state $\tau_q$) and is therefore one-to-one; in fact it extends to an isomorphism $L^2(\Gamma_q(H),\tau_q)\to \mathcal{F}_q(H)$.  In view of \eqref{eq.Parth.id}, we refer to this inverse map as the {\bf $q$-Wigner stochastic integral $I_q$}.  The image $I_q(H^{\otimes n})$ of the $n$-particle space in $\mathcal{F}_q(H)$ under this map is an operator subspace $\Gamma_n^q(H)$ of $\Gamma_q(H)$, which is identified as the $q$-Wigner chaos of order $n$.  It is possible to construct the map $I_q$ directly as a kind of stochastic integral with respect to the $q$-Brownian motion process in $\Gamma_q(L^2(\mathbb{R}_+))$, cf.\ \cite{MR1952458}.  In the special case $q=0$, a very rich theory of free stochastic integration has been developed, for example in \cite{MR1660906}.  See sections \ref{sect.Wigner.chaos} and \ref{sect.q-Wigner.chaos} for details.

Returning to the Fourth Moment Theorem: in 2012, the first author together with Peccati, Nourdin, and Speicher proved a complete analog of Theorem \ref{thm.NuPec} in the free probability setting.

\begin{thm}[K, Nourdin, Peccati, Speicher, \cite{1529a405-f4b1-33eb-aeac-05c9f76e60d3}] \label{thm.KNPS} Let $f_k$ be a sequence of mirror symmetric functions in $L^2(\mathbb{R}^n_+)$ (i.e.\ $f(t_1,\ldots,t_n) = \overline{f(t_n,\ldots,t_1)}$) with $\|f_k\|_2\to 1$ as $k\to\infty$.  Then the following three conditions are equivalent.
\begin{enumerate}
    \item $\displaystyle{\lim_{k\to\infty}\tau_0[I_0(f_k)^4] = 2}$.
    \item $\displaystyle{\lim_{k\to\infty} f_k \overset{p}{\frown} f_k = 0}$ in $L^2(\mathbb{R}_+^{2(n-p)})$ for $p=1,\ldots,n-1$.
    \item The sequence $(I_0(f_k))_{k\in\N}$ converges in distribution to a standard semicircular random variable.
\end{enumerate}
\end{thm}

Shortly after, Deya, Norredine, and Nourdin partly generalized these results to the case $q\in[0,1]$.  Interestingly, the limit distribution for general $q$ is not the same for all orders of chaos.
\begin{thm}[Deya, Norredine, Nourdin, \cite{MR3089666}] \label{thm.DNN} Let $f_k$ be a sequence of {\bf fully} symmetric functions in $L^2(\mathbb{R}^n_+)$ with $\|f_k\|_2\to 1$ as $k\to\infty$.  Then the following three conditions are equivalent.
\begin{enumerate}
    \item $\displaystyle{\lim_{k\to\infty}\tau_q[I_q(f_k)^4] = 2+q^{n^2}}$.
    \item $\displaystyle{\lim_{k\to\infty} f_k \overset{p}{\frown} f_k = 0}$ in $L^2(\mathbb{R}_+^{2(n-p)})$ for $p=1,\ldots,n-1$.
    \item The sequence $(I_q(f_k))_{k\in\N}$ converges in $\ast$-distribution to a $q^{n^2}$-Gaussian random variable.
\end{enumerate}
\end{thm}
Note that, as in the $q=0$ case proved in \cite{1529a405-f4b1-33eb-aeac-05c9f76e60d3}, the most natural general assumption on the integral kernels $f_k$ is mirror symmetry: this is precisely equivalent to the stochastic integral $I_q(f_k)$ being selfadjoint.  At this time, it remains unknown whether Theorem \ref{thm.DNN} holds for mirror symmetric functions, not just the more restrictive class of fully symmetric functions.  Moreover, it remains unclear if the theorem can be generalized to $q<0$ (although \cite{MR3089666} gives a counterexample with $q=-\frac12$, it is unknown whether there are counterexamples for any other $q<0$).

All of the above results are for ($q$-)Gaussian approximation in a {\bf single} order of chaos.  A more general and natural question is what happens for random variables with {\em finite} chaos expansions: i.e.\ in a linear combination of chaos of different orders.  Remarkably, this question was not addressed at all until the recent paper \cite{basseoconnor2025fourthmoment} considered the question of a combination of {\em two} orders of chaos.  Their main result was that the fourth moment theorem holds in essentially the same manner, provided the two orders of chaos have opposite parity; otherwise, the theorem fails.  To properly state their result, we reframe condition (1) from the above theorems in terms of the fourth {\em cumulant} rather than the fourth {\em moment}.  For a centered random variable $X$, its fourth (classical) cumulant $c_4$ is
\[ c_4(X) = \mathbb{E}[X^4]-3(\mathbb{E}[X^2])^2. \]

\begin{thm}[Basse-O'Connor, Kramer-Bang, Svedsen, \cite{basseoconnor2025fourthmoment}] \label{thm.BOKBS25}
Let $m,n \in\N$ have opposite parities.  For each $k\in\N$, let $f_k \in L^2(\mathbb{R}^m_+)$ and $g_k \in L^2(\mathbb{R}^n_+)$ be fully symmetric functions, and let $X_k = I^W(f_k)$ and $Y_k=I^W(g_k)$.  Assume $\mathbb{E}[X_k^2] + \mathbb{E}[Y_k^2] = \sigma^2$ for all $k\in\N$.  Then the following three conditions are equivalent.
\begin{enumerate}
    \item $\displaystyle{\lim_{k\to\infty} c_4[X_k+Y_k] = 0}$.
    \item $\displaystyle{\lim_{k\to\infty} f_k \overset{p}{\frown} f_k = 0}$ in $L^2(\mathbb{R}_+^{2(m-p)})$ for $p=1,\ldots,m-1$,  and $\displaystyle{\lim_{k\to\infty} g_k \overset{p}{\frown} g_k = 0}$ in $L^2(\mathbb{R}_+^{2(n-p)})$ for $p=1,\ldots,n-1$.
    \item The sequence $(X_k+Y_k)_{k\in\N}$ converges in distribution to a centered Gaussian random variable with variance $\sigma^2$.
\end{enumerate}
\end{thm}

Since $c_4[X_k+Y_k] = \mathbb{E}[(X_k+Y_k)^4]-3(\mathbb{E}[(X_k+Y_k)^2])^2$, and since $\mathbb{E}[(X_k+Y_k)^2] = \mathbb{E}[X_k^2]+\mathbb{E}[Y_k^2] = \sigma^2$ (because $X_k,Y_k$ are in different orders of chaos and are therefore orthogonal in $L^2$), condition (1) above is equivalent to the condition in the preceding theorems scaled by variance: the assumption is that the fourth moment $\mathbb{E}[(X_k+Y_k)^4]$ converges to $3\sigma^4$, which is the fourth moment of a Gaussian $\mathcal{N}(0,\sigma^2)$ random variable.  Theorem \ref{thm.BOKBS25} is stated in terms of the fourth cumulant because the key innovation of the paper is the realization that, for Winer It\^o integrals $X,Y$ of different parities, the following identity holds:
\begin{equation} \label{eq.c4.id} c_4[X+Y] = c_4[X]+c_4[Y] + 6\mathrm{Cov}[X^2,Y^2]. \end{equation}
Equation \ref{eq.c4.id} and Theorem \ref{thm.1} are proved in Section \ref{0-case}.

\subsection{Main Results}

In this paper, we prove the precise analogs of Theorem \ref{thm.BOKBS25} in free Wigner chaos (as in Theorem \ref{thm.KNPS}) and in $q$-Wigner chaos (as in Theorem \ref{thm.DNN}).  We begin with the free case.  Here the statement is in terms of the fourth {\em free} cumulant $\kappa_4$, defined (for a mean $0$ random variable) as $\kappa_4[X] = \tau[X^4]-2(\tau[X^2])^2$.

\begin{thm} \label{thm.1} 
Let $m,n\in\N$ have opposite parities.  For each $k\in\N$, let $f_k \in L^2(\mathbb{R}^m_+)$ and $g_k \in L^2(\mathbb{R}^n_+)$ be mirror symmetric functions, and let $X_k = I_0(f_k)$ and $Y_k=I_0(g_k)$ be the free Wigner integrals in $\Gamma_0(H)$ where $H=L^2(\mathbb{R}_+)$.  Assume $\tau_0[X_k^2] + \tau_0[Y_k^2] \to \sigma^2$ as $k\to\infty$.  Then the following three conditions are equivalent.
\begin{enumerate}
    \item $\displaystyle{\lim_{k\to\infty} \kappa_4[X_k+Y_k] = 0}$.
    \item $\displaystyle{\lim_{k\to\infty} f_k \overset{p}{\frown} f_k = 0}$ in $L^2(\mathbb{R}_+^{2(m-p)})$ for $p=1,\ldots,m-1$ and $\displaystyle{\lim_{k\to\infty} g_k \overset{p}{\frown} g_k = 0}$ in $L^2(\mathbb{R}_+^{2(n-p)})$ for $p=1,\ldots,n-1$.
    \item The sequence $(X_k+Y_k)_{k\in\N}$ converges in distribution to a centered semicircular random variable with variance $\sigma^2$.
\end{enumerate}
\end{thm}

As with Theorem \ref{thm.BOKBS25}, a key first step is a result expressing a polarized identity for the fourth free cumulant of a sum.  In this setting, the exact analog of \eqref{eq.c4.id} does {\em not} hold. Rather, we show below that if $X,Y$ are free Wigner stochastic integrals of different parities, then
\begin{equation} \label{eq.k4.id} \kappa_4[X+Y] = \kappa_4[X]+\kappa_4[Y] + 4\mathrm{Cov}[X^2,Y^2]+2\tau_0[XYXY]. \end{equation}

We also prove the $q$-Wigner chaos analog of Theorem \ref{thm.1}, for $q\in[0,1)$, with comparable assumptions as in Theorem \ref{thm.DNN}.  As above, they key starting point is a $q$-analog of \eqref{eq.c4.id}, which (like the theorem itself) could be stated in terms of a $q$-analog of cumulants interpolating between free $k_4$ and classical $c_4$; for the analog of \eqref{eq.c4.id} see Lemma \ref{q-keylemma}.  In this case, we will state the condition in the theorem more directly.

\begin{thm} \label{thm.2} Let $q\in[0,1)$, and let $m,n \in\N$ have opposite parities.  Let $H=L^2(\mathbb{R}_+)$.  For each $k\in\N$, let $f_k \in L^2(\mathbb{R}^m_+)\cong H^{\odot m}$ and $g_k \in L^2(\mathbb{R}^n_+)\cong H^{\odot n}$ be fully symmetric functions, and let $X_k = I_q(f_k)$ and $Y_k=I_q(g_k)$ be the $q$-Wigner integrals in $\Gamma_q(H)$. Assume $\|f_k\|_{\mathcal{F}_q(H)}$ and $\|g_k\|_{\mathcal{F}_q(H)}$ converge and define
\[ \sigma_X = \lim_{k\to\infty}\|f_k\|_{\mathcal{F}_q(H)} \quad \& \quad \sigma_Y = \lim_{k\to\infty}\|g_k\|_{\mathcal{F}_q(H)}. \]
Then the following three conditions are equivalent.
\begin{enumerate}
    \item $\displaystyle{\lim_{k\to\infty} \tau_q[(X_k+Y_k)^4] = (2+q^{m^2})\sigma_X^4 + (2+q^{n^2})\sigma_Y^4 + (4+2q^{mn})\sigma_X^2\sigma_Y^2}$.
    \item $\displaystyle{\lim_{k\to\infty} f_k \overset{p}{\frown} f_k = 0}$ in $L^2(\mathbb{R}_+^{2(m-p)})$ for $p=1,\ldots,m-1$ and $\displaystyle{\lim_{k\to\infty} g_k \overset{p}{\frown} g_k = 0}$ in $L^2(\mathbb{R}_+^{2(n-q)})$ for $p=1,\ldots,n-1$.
    \item The sequence $(X_k+Y_k)_{k\in\N}$ converges in distribution to a mixed $Q$-Gaussian random variable with respect to the vector $[\sigma_X,\sigma_Y]$, where
    \[ Q = \begin{bmatrix} q^{m^2} & q^{mn} \\ q^{mn} & q^{n^2} \end{bmatrix}. \]
\end{enumerate}
\end{thm}
The proof of Theorem \ref{thm.2}, along with the definition of mixed $Q$-Gaussian random variables (generalizing the distribution of the field operators on $q$-Fock spaces) can be found in Section \ref{q-case}.

\begin{rem} \begin{enumerate} \item Later generalizations of Theorem \ref{thm.NuPec} added another equivalent condition stated in terms of Malliavin derivatives, which leads to a concise quantitative estimate for the Wasserstein distance to the Gaussian distribution in terms of the fourth cumulant.  In the free case, a version of this leg of the Fourth Moment Theorem was proved in \cite{1529a405-f4b1-33eb-aeac-05c9f76e60d3} in the special case of fully symmetric kernel functions $f_k$, with the full proof completed later in \cite{Cebron2021}.  Similarly, \cite{basseoconnor2025fourthmoment} proved a comparable Malliavin calculus condition which, following the same framework, yields a quantitative bound for the distance to a Gaussian distribution.  We have not explored this stochastic analytic connection to our main Theorems \ref{thm.1} and \ref{thm.2} yet, leaving them for a future publication.
\item The condition that the kernel functions $f_k,g_k$ should be fully symmetric is stronger than natural in Theorems \ref{thm.DNN} and \ref{thm.2}, where mirror symmetry is equivalent to selfadjointness of the stochastic integrals $X_k,Y_k$.  For a single or double order of $q$-Wigner chaos, it is yet unknown how to generalize Fourth Moment Theorems to that fully general arena.
\end{enumerate}
\end{rem}

\section{Preliminaries}\label{Preliminaries}
\subsection{Semicircle distribution and free cumulants}
In this section, we briefly introduce basic notions in free probability. We refer \cite{MR2266879} for details.
The semicircle distribution $S(0,\sigma^2)$ with mean $0$ and variance $\sigma^2$ is defined by its density function
\[  \frac{1}{2\pi \sigma^2}\sqrt{4\sigma^2-x^2}\ 1_{[-2\sigma,2\sigma].} \]
This distribution is also characterized by its moments 
\[E(s^n)=\begin{cases}
    \frac{1}{k+1}\binom{2k}{k}  &(n=2k) \\
   \quad  0 \quad &(n=2k-1).
\end{cases} \]
In free probability, we say that non-commutative random variables $X,Y$ (with a state $\tau$) are freely independent if, for any polynomials $P_1,\ldots,P_m$ and $Q_1,\ldots,Q_m$, we have
    \[ \tau\left[\overset{\circ}{P_1(X)}\overset{\circ}{Q_1(Y)}\cdots \overset{\circ}{P_m(X)}\overset{\circ}{Q_m(Y)}\right]=0\]
    where we set $\overset{\circ}{a}=a-\tau(a)$.
    The free independence is also characterized by free cumulants $\{\kappa_n\}_{n\in \mathbb{N}}$, which are multi-linear functionals defined inductively by the moment-cumulant formula in terms of non-crossing partitions (see \cite[Lecture 11]{MR2266879}). Actually, it is known that 
    $X$ and $Y$ are freely independent if and only if mixed free cumulants vanish (see \cite[Theorem 11.16]{MR2266879}). 
    By the moment-cumulant formula, the free cumulants of a random variable characterize its moments, and the free cumulants of the semicircle distribution $S(0,\sigma^2)$ are
\[ \kappa_2(s)=\sigma^2, \quad \kappa_n(s)=0 \  (n \neq 2)\]
where $\kappa_n(s)=\kappa_n(s,\ldots,s)$. By using the equivalence of free independence and vanishing mixed free cumulants, we can see that, if $X,Y$ are freely independent semicircle distributions $S(0,\sigma_X^2)$ and $S(0,\sigma_Y^2)$, then $X+Y$ is also a semicircle distribution $S(0,\sigma_X^2+\sigma_Y^2)$.  
We also use the following formula of the fourth cumulant for a centered random variable $X$, 
\[\kappa_4(X)=\tau(X^4)-2\tau(X^2)^2.\]

\subsection{Wigner chaos\label{sect.Wigner.chaos}}
In this section, we introduce our basic tools of Wigner chaos. We refer \cite[Section 5]{MR1660906} for more details. Let $H_{\mathbb{R}}$ be a separable real Hilbert sapce and $H=H_{\mathbb{R}}+\sqrt{-1}H_{\mathbb{R}}$ be the complexification $H_{\mathbb{R}}$. We define the full Fock space by
\[\mathcal{F}_0(H)=\mathbb{C}\Omega \oplus \bigoplus_{n=1}^{\infty}H^{\otimes n}\]
where $\Omega$ is a unit vector. 
For $ f \in H_{\mathbb{R}}$, we consider left creation operator $l( f)$ defined by
\[c_0(f) f_1\otimes  f_2 \cdots \otimes  f_n= f\otimes  f_1\otimes  f_2 \cdots \otimes  f_n,\]
and the von Neumann algebra generated by semicircular operators 
\[\Gamma_0(H_{\mathbb{R}})=\{\omega_0( f)=c_0(f)+c_0(f)^*|\  f \in H_{\mathbb{R}}\}^{"}.\] 
It is known that $\Gamma_0(H_{\mathbb{R}})$ admits a unique faithful normal tracial state $\tau_0$ deinfed by
\[ \tau_0(X)=\langle X\Omega,\Omega\rangle,\]
and each $\omega_0( f)$ has semicircle distribution $S(0,\| f\|^2)$ with respect to $\tau_0$. Moreover, if $ f\perp  g$, it is known that $\omega_0( f)$ and $\omega_0( g)$ are freely independent (e.g., \cite{MR1217253}).   

Let $\mathcal{F}_{alg}(H)$ denote the algebraic Fock space of $H$, i.e., the linear subspace of $\mathcal{F}_0(H)$ that consists of finite linear spans of the algebraic tensor products of $H$.
Then, we can construct the linear map $I_0: \mathcal{F}_{alg}(H) \to \Gamma_0(H_{\mathbb{R}})$ determined by
\[I_0( f)\Omega= f.\]
Note that $ I_0$ extends to Hilbert space isomorphism between the full Fock space $\mathcal{F}_0(H)$ and GNS Hilbert space $L^2(\Gamma_0(H_{\mathbb{R}}),\tau_0)$ with respect to $\tau_0$.
In particular, when $  f \in H^{\otimes n}$, $ I_0( f)$ is called the $n$th free Wigner chaos. We define the complex conjugate of $ f=x+iy\in H$
by $\overline{ f}=x-iy $ and $*: \mathcal{F}_0(H)\to \mathcal{F}_0(H)$ be anti-linear map defined by the extension of 
\[( f_1\otimes  f_2 \otimes \cdots \otimes  f_n)^*=\overline{ f_n}\otimes \overline{ f_{n-1}} \otimes \cdots \otimes \overline{ f_1}.\]
In Section \ref{0-case}, we use the following contraction formula \cite[Proposition 5.3.3]{MR1660906} for $ f \in H^{\otimes m}$ and $ g \in H^{\otimes n}$
\[ I_0( f) I_0( g)=\sum_{k=0}^{m\wedge n} I_0( f\overset{k}{\frown}  g).\]
In this formula, we set
\[ f\overset{k}{\frown}  g := (1_{m-k}\otimes \Phi_k \otimes 1_{n-k})( f \otimes  g)\in H^{\otimes m+n-2k}\]
where $1_m\in B(H^{\otimes m})$ is the identity operator and  $\Phi_k:H^{\otimes k}\otimes H^{\otimes k} \to \mathbb{C}$ is a linear map defined by $\Phi_k( f\otimes  g)=\langle  f,  g^*\rangle$. 

Our particular interest is the case where $H_{\mathbb{R}}$ is the set of real-valued $L^2$ functions on $\mathbb{R}_+:=[0,\infty)$ with respect to the Lebesgue measure and we set $H=L^2(\mathbb{R_+})$ for complex-valued $L^2$-functions on $\mathbb{R}_+$. In this case, we have a Hilbert space isomorphism
\[ L^2(\mathbb{R}_+)^{\otimes n}\cong L^2(\mathbb{R}_+^n).\]Through this identification, we have $ f(t_1,t_2,\ldots,t_n)^*=\overline{ f(t_n,t_{n-1},\ldots,t_1)}$ and
\[( f\overset{k}{\frown} g)(\vec{t}_{m-k},\vec{r}_{n-k}))= \int_{\mathbb{R}_+^k} f(\vec{t}_{m-k}, \vec{s}_k) g(\vec{s}_k^*,\vec{r}_{n-k})  d\vec{s}_k\]
where $\vec{t}_{m-k}=(t_1,\ldots,t_{m-k})$, $\vec{s}_k=(s_1,\ldots,s_k)$, $\vec{r}_{n-k}=(r_1,\ldots,r_{n-k}) $, and $\vec{s}_k^*=(s_k,\ldots,s_1)$. 

 We also use the following identity (see \cite[Lemma 1.24]{1529a405-f4b1-33eb-aeac-05c9f76e60d3})
\[ ( f\overset{k}{\frown} g)^*= g^*\overset{k}{\frown}  f^*.\]

\subsection{$q$-Wigner chaos\label{sect.q-Wigner.chaos}}
We also consider a $q$-deformation of Wigner chaos. We refer \cite{MR1463036,MR3749579,MR1952458} for details.  The difference from the previous section is that we consider the inner product on $H^{\otimes m}$ composed with a ``$q$-symmetrization''
\[\langle  f_1\otimes \cdots \otimes  f_m,  g_1 \otimes \cdots \otimes  g_n \rangle_q = \delta_{m,n} \langle  f_1\otimes \cdots \otimes  f_m, P^{(m)}_q  g_1 \otimes \cdots \otimes  g_n\rangle_{H^{\otimes m}} \]
where the operator $P^{(m)}_q$ on $H^{\otimes m}$ is defined by
\begin{align*} P_q^{(m)}( f_1\otimes \cdots \otimes  f_m) 
&= \sum_{\pi \in S_m} q^{\mathrm{inv}(\pi)}  f_{\pi(1)}\otimes \cdots \otimes  f_{\pi(m)}\\
&=\sum_{\pi \in S_m} q^{\mathrm{inv}(\pi)}\  f_{\pi^{-1}(1)}\otimes \cdots \otimes  f_{\pi^{-1}(m)}.\end{align*}
The $q$-Fock space $\mathcal{F}_q(H)$ is the Hilbert space obtained from $\mathcal{F}_{\mathrm{alg}}(H)$ and this $q$-deformed inner product.
For each $\pi \in S_m$, we define the permutation operator $\pi$ on $H^{\otimes m}$ by $\pi( f_1\otimes \cdots \otimes  f_m)= f_{\pi^{-1}(1)}\otimes \cdots \otimes  f_{\pi^{-1}(m)}$. Then, we have $P_q^{(m)}=\sum_{\pi\in S_m}q^{\mathrm{inv}(\pi)}\pi=\sum_{\pi\in S_m}q^{\mathrm{inv}(\pi)}\pi^{-1}$. 
It is convenient to introduce the operator $R_{k,m}$ on $H^{\otimes m}$ with $k\le m$ by
\[R_{k,m}=\sum_{\pi \in \quotient{S_m}{S_k \times S_{m-k}}}q^{\mathrm{inv}(\pi)} \pi\]
where
\[ \quotient{S_m}{S_k \times S_{m-k}} :=\text{ the left cosets of $S_k \times S_{m-k}$ in $S_m$.} \]
In this paper, we always take the unique representative of $\sigma \in \quotient{S_m}{S_k \times S_{m-k}}$ so that $\mathrm{inv}(\sigma)$ is minimal.
It is known that $R_{k,n}$ satisfies the following property (\cite[Theorem 2.1]{MR1811255})
\[R_{k,m}(P_q^{(k)}\otimes P_q^{(m-k)})=P_q^{(m)}=(P_q^{(k)}\otimes P_q^{(m-k)})R_{k,m}^*.\]
In the same way, we define the left creation operator $c_q( f)$, and we use the notation $\Gamma_q(H_{\mathbb{R}})$ for the von Neumann algebra generated by $\omega_q(f)=c_q(f)+c_q(f)^*$ with $ f \in H_{\mathbb{R}}$. It is known that the vacuum state $\tau_q(X)=\langle X \Omega,\Omega\rangle_q $ is still a faithful normal trace on $\Gamma_q(H_{\mathbb{R}})$. Moreover, there is a Hilbert space isomorphism $I_q: \mathcal{F}_q(H)\to L^2(\Gamma_q(H_{\mathbb{R}}),\tau_q)$ such that
\[I_q( f)\Omega = f.\]
We also have the contraction formula for $ f\in H^{\otimes m}$ and $ g\in H^{\otimes n}$,
\[ I_q( f)I_q( g)=\sum_{k=0}^{m\wedge n} I_q( f\overset{k}{\frown}_q g).\]
In this case, we set (see \cite{MR3749579} and \cite[Lemma 4.6]{MR4251282})
\[  f\overset{k}{\frown}_q g= (1_{m-k}\otimes \Phi^q_k \otimes 1_{n-k})(R^*_{m-k,m} f) \otimes (R^*_{k,n} g)\in H^{\otimes m+n-2k} \]
where $\Phi^q_k:H^{\otimes k}\otimes H^{\otimes k} \to \mathbb{C}$ is a linear map defined by $\Phi_k^q( f\otimes  g)=\langle  f,  g^*\rangle_q$.
 Similar to $q=0$ case, we use the following identity (\cite[Lemma 2.5]{MR3749579})
   \[ ( f \overset{k}{\frown}_q g)^*= g^* \overset{k}{\frown}_q f^*.\]
   
Now, we fix $H=L^2(\mathbb{R}_+)$.
  In this case, for $f=f_1\otimes f_2 \otimes \cdots \otimes f_m \in H^{\otimes m}$ and $\pi \in S_m$, we have $\pi (f)=f_{\pi^{-1}(1)}\otimes f_{\pi^{-1}(2)}\otimes \cdots \otimes f_{\pi^{-1}(m)}$ and
   \[ f_{\pi^{-1}(1)}(t_1) f_{\pi^{-1}(2)}(t_2)\cdots  f_{\pi^{-1}(m)}(t_m)=f_{1}(t_{\pi(1)}) f_{2}(t_{\pi(2)})\cdots  f_{m}(t_{\pi(m)}).\]
   Under the isomorphism \[L^2(\mathbb{R}_+)^{\otimes m}\ni f=\otimes_{i=1}^mf_i\mapsto f(\vec{t})=\prod_{i=1}^mf_i(t_i)\in L^2(\mathbb{R}^m_+),\] we have
   \[\pi(f)(\vec{t})=f(\pi^{-1}(\vec{t}))\]
   where we define $\pi^{-1}(\vec{t})=(t_{\pi(1)},t_{\pi(2)},\ldots,t_{\pi(m)})$. 
We say $f \in L^2(\mathbb{R}_+)^{\otimes m} \cong L^2(\mathbb{R}_+^m)$ is fully symmetric if $f$ is real-valued and $\pi(f)=f$ for any $\pi \in S_m$. Note that for a fully symmetric function $f\in L^2(\mathbb{R}_+)^{\otimes m }$, we have $P_q^{(m)}f=[m]_q! f$ and 
\[ \|I_q(f)\|_{2}=\|f\|_{\mathcal{F}_q(L^2(\mathbb{R}_+))}=\sqrt{[m]_q!} \ \|f \|_{L^2(\mathbb{R}_+)^{\otimes m}}\]
where $[m]_q!=[m]_q[m-1]_q!\cdots [2]_q[1]_q$
with $[m]_q=\sum_{i=1}^{m}q^{i-1}$.
We also remark that, when $ f\in L^2(\mathbb{R}_+)^{\otimes m}$ and $ g\in L^2(\mathbb{R}_+)^{\otimes n}$ are fully symmetric, we have a simple expression for $ f\overset{k}{\frown}_q g$ (\cite[Theorem 4.2]{MR3749579}),
\[ f\overset{k}{\frown}_q g=[k]_q! \binom{m}{k}_q\binom{n}{k}_q  f\overset{k}{\frown} g \]
where $\binom{m}{k}_q=\frac{[m]_q!}{[k]_q![m-k]_q!}$. One can check this by $R^*_{m-k,m} f=\binom{m}{k}_q f$, $ R^*_{k,n} g=\binom{n}{k}_q  g$, and
\[ (1_{m-k}\otimes \Phi_k^q \otimes 1_{n-k}) f \otimes  g = [k]_q! (1_{m-k}\otimes \Phi_k \otimes 1_{n-k}) f \otimes  g. \]

\section{Proof for Wigner chaos}\label{0-case}
We start with proving the key lemma, which is a free probabilistic analog of \cite[Lemma 1.7]{basseoconnor2025fourthmoment}. In this section and the next section, we set $H=L^2(\mathbb{R}_+)$.
\begin{lm}\label{keylemma}
   Let $X= I_0( f)$ and $Y= I_0( g)$ with $ f= f^*\in H^{\otimes m}$ and $ g = g^*\in H^{\otimes n}$ with $m,n\ge 1$. If $m,n$ have opposite parities, we have
    \[\kappa_4(X+Y)=\kappa_4(X)+\kappa_4(Y)+4\mathrm{cov}(X^2,Y^2)+2\tau_0(XYXY).\]
    Moreover, we have
    \[\kappa_4(X), \ \kappa_4(Y),\ 4\mathrm{cov}(X^2,Y^2)+2\tau_0(XYXY)\ge 0.\]
    \end{lm}

\begin{proof}
Since $X+Y$ is centered, we have
\[\kappa_4(X+Y)=\tau_0[(X+Y)^4]-2\tau_0[(X+Y)^2]^2=\tau_0[(X+Y)^4]-2\left(\tau_0(X^2)+\tau_0(Y^2)\right)^2.\]
In the expansion of $\tau_0[(X+Y)^4]$, $\tau_0(X^3Y)$ and $\tau_0(XY^3)$ are equal to $0$ since, by the contraction formula, each Wigner chaos in the chaos expansion of $X^3Y$ (resp. $Y^3X$) has degree $3m+n-2k$ (resp. $m+3n-2k$) for some integer $k$ which cannot be equal to $0$ if $m,n$ have opposite parties.  
 Thanks to the trace property of $\tau_0$, we have
\[\tau_0[(X+Y)^4]=\tau_0(X^4)+\tau_0(Y^4)+4\tau_0(X^2Y^2)+2\tau_0(XYXY)\]
and
\begin{align*}
\kappa_4(X+Y)&=\tau_0[(X+Y)^4]-2\left(\tau_0(X^2)+\tau_0(Y^2)\right)^2\\
&=\kappa_4(X)+\kappa_4(Y)+4\tau_0(X^2Y^2)-4\tau_0(X^2)\tau_0(Y^2)+2\tau_0(XYXY)\\
&=\kappa_4(X)+\kappa_4(Y)+4\mathrm{cov}(X^2,Y^2)+2\tau_0(XYXY).
\end{align*}
For the positivity of $\kappa_4(X)$ ($\kappa_4(Y)$ as well) which is proved in \cite[Theorem 1.6]{1529a405-f4b1-33eb-aeac-05c9f76e60d3}, we apply the contraction formula to $X^2$ and we have
\begin{align*}
    X^2&=\sum_{k=0}^m I_0( f\overset{k}{\frown}  f).
\end{align*}
By taking $L^2$-norm, we have
\begin{align*}
    \tau_0(X^4)&=\sum_{k=0}^m\| f\overset{k}{\frown}  f\|^2.
\end{align*}

Since $\| f\overset{0}{\frown} f\|^2=\| f\|^4=\| f\overset{m}{\frown} f\|^2$, we have
\[\kappa_4(X)=\tau_0(X^4)-2\tau_0(X^2)=\sum_{k=1}^{m-1}\| f\overset{k}{\frown} f\|^2\ge 0.\]
For the positivity of $4\mathrm{cov}(X^2,Y^2)+2\tau_0(XYXY)$, we apply the contraction formula to $XY$ and we obtain 
  \[ XY=\sum_{k=0}^{m\wedge n} I_0( f\overset{k}{\frown} g).\]

  Since $\tau_0$ is a trace and $\left\| f\overset{0}{\frown}  g\right\|^2=\left\| f\otimes  g\right\|^2=\| f\|^2\| g\|^2$, we have
   \begin{align*}
       4\mathrm{cov}(X^2,Y^2)&=4\tau_0(YX^2Y)-4\tau_0(X^2)\tau_0(Y^2)\\
       &=4\|XY\|_2^2 -4 \|X\|_2^2\|Y\|_2^2 \\
       &=4\sum_{k=0}^{m\wedge n}\left\| f\overset{k}{\frown}  g\right\|^2-4\| f\|^2\| g\|^2\\
       &=4\sum_{k=1}^{m\wedge n}\left\| f\overset{k}{\frown}  g\right\|^2.
   \end{align*}
  Since $\tau_0(XYXY)$ is real, the contraction formula of $XY$ also tells us 
\[2\tau_0(XYXY)=\sum_{k=0}^{m\wedge n}\left\langle  f\overset{k}{\frown}  g , g\overset{k}{\frown}  f    \right\rangle +\left\langle  g\overset{k}{\frown}  f, f\overset{k}{\frown}  g     \right\rangle=2\mathrm{Re}\sum_{k=0}^{m\wedge n}\left\langle  f\overset{k}{\frown}  g , g\overset{k}{\frown}  f    \right\rangle.  \]
   Since $\ast$ is the anti-unitary and $( f\overset{k}{\frown}  g)^*= g^* \overset{k}{\frown} f^*= g\overset{k}{\frown}  f$ , we have
   \begin{align*}
&4\mathrm{cov}(X^2,Y^2) +2\tau_0(XYXY)\\
&= 4\sum_{k=1}^{m\wedge n}\left\| f\overset{k}{\frown}  g\right\|^2 +2\mathrm{Re} \sum_{k=0}^{m\wedge n}\left\langle  f\overset{k}{\frown}  g , g\overset{k}{\frown}  f    \right\rangle \\
&=2\sum_{k=1}^{m\wedge n-1}\left\| f\overset{k}{\frown}  g\right\|^2 +\sum_{k=1}^{m\wedge n}\left\| f\overset{k}{\frown}  g+ g\overset{k}{\frown}  f\right\|^2 +2\| f\overset{m\wedge n}{\frown} g\|^2+ 2\mathrm{Re}\langle  f\otimes  g, g\otimes  f\rangle.
   \end{align*}
   The positivity of $4\mathrm{cov}(X^2,Y^2)+2\tau_0(XYXY)$ follows from the following identity
   \[2\| f\overset{m\wedge n}{\frown} g\|^2+ 2\mathrm{Re}\langle  f\otimes  g, g\otimes  f\rangle = \| f\overset{m\wedge n}{\frown}  g+ g\overset{m\wedge n}{\frown}  f\|^2.\]
   To see this, we may assume $m>n$ and show
   \[\langle f\otimes  g, g\otimes  f\rangle=\langle  g\overset{n}{\frown} f, f\overset{n}{\frown} g\rangle. \]
   We can compute left hand side as follows. 
    \begin{align*}
       \langle  f \otimes  g,  g \otimes  f\rangle=\int_{\mathbb{R}_+^{m+n}} f(\vec{s},\vec{t}) g(\vec{r})\cdot \overline{ g(\vec{s}) f(\vec{t},\vec{r})} \ d\vec{s}d\vec{t}d\vec{r}. 
   \end{align*}
   Since $ g= g^*$, this integral is equal to
   \[\int_{\mathbb{R}_+^{m-n}}d\vec{t}\int_{\mathbb{R}_+^n} g(\vec{s}^*) f(\vec{s},\vec{t})d\vec{s} \cdot \int_{\mathbb{R}_+^n}\overline{ f(\vec{t},\vec{r}) g(\vec{r}^*)} d\vec{r}, \]
   which is equal to $\langle  g\overset{n}{\frown} f, f\overset{n}{\frown} g\rangle$.
   \end{proof}

  \begin{thm}\label{fourthmoment}
     Let $X_k= I_0( f_k)$ and $Y_k= I_0( g_k)$ with $ f_k=f_k^* \in H^{\otimes m}$ and $ g_k=g_k^* \in H^{\otimes n}$ where $\tau_0(X_k^2)+\tau_0(Y_k^2)\to \sigma^2$ as $k\to \infty$ and $m,n\in \mathbb{N}$ have opposite parities.  
     Then the following are equivalent: \begin{enumerate}
          \item $\lim_{k\to \infty}\kappa_4(X_k+Y_k)=0$.
          \item  $\displaystyle{\lim_{k\to\infty} f_k \overset{p}{\frown} f_k = 0}$ in $H^{\otimes 2(m-p)}$ for $p=1,\ldots,m-1$ and $\displaystyle{\lim_{k\to\infty} g_k \overset{p}{\frown} g_k = 0}$ in $H^{\otimes 2(n-p)}$ for $p=1,\ldots,n-1$.
         \item $X_k+Y_k$ converges in distribution toward $ S(0,\sigma^2)$. 
     \end{enumerate}
      
  \end{thm}
  \begin{proof}
  
     Note that since $\{X_k+Y_k\}_{k\in \mathbb{N}}$ are uniformly bounded operators on $n$ (i.e., their spectral measures are compactly supported) by the Haagerup inequality $\|X_k+Y_k\|\le c \|X_k+Y_k\|_2$ for some constant $c>0$ (for example \cite[Theorem 5.3.4]{MR1660906}), convergence in moments is equivalent to convergence in distribution. Since a semicircle distribution $s$ satisfies $\kappa_4(s)=0$, we see that (3) implies (1). By Lemma \ref{keylemma}, if $\lim_{k\to \infty}\kappa_4(X_k+Y_k)=0$, then we have \[\lim_{k\to \infty}\kappa_4(X_k)=\lim_{k\to \infty}\kappa_4(Y_k)=0.\] Since $\kappa_4(X_k)=\sum_{p=1}^{m-1} \|f_k\overset{p}{\frown}f_k\|^2$ in the proof of Lemma \ref{keylemma}, we have $\lim_{k\to \infty}f_k\overset{p}{\frown}f_k=0 $ in $H^{\otimes 2(m-p)}$ for $p=1,\ldots,m-1$. Similarly, we have $\lim_{k\to \infty}g_k\overset{p}{\frown}g_k=0 $ in $H^{\otimes 2(m-p)}$ for $p=1,\ldots,m-1$, which shows (1) implies (2). To see (2) implies (3), it is enough to check if all convergent subsequences of  $\{X_k+Y_k\}_{k\in \mathbb{N}}$ converge in distribution to the $S(0,\sigma^2)$ by uniform boundedness. Let $\{X_{k_l}+Y_{k_l}\}_{
      l\in \mathbb{N}}$ be a convergent subsequence. Since $\tau_0(X_{k_l}^2)+\tau_0(Y_{k_l}^2)\to \sigma^2$ as $l \to \infty$, we can take convergent subsequence of $\{\tau_0(X_{k_l}^2)\}_{l\in \mathbb{N}}$ and $\{\tau_0(Y_{k_l}^2)\}_{l\in \mathbb{N}}$. We use the same notation $\{\tau_0(X_{k_l}^2)\}_{l\in \mathbb{N}}$, $\{\tau_0(Y_{k_l}^2)\}_{l\in \mathbb{N}}$ for convergent subsequesnces, and we denote their limit  by $\sigma^2_X$, $\sigma_Y^2$.  
      Note that we have $\tau_0(X_{k_l}Y_{k_l})=0$ for any $l$ since their degrees are different. If we assume (2), which is equivalent to $\lim_{l\to \infty}\kappa_4(X_{k_l})=\lim_{l\to \infty}\kappa_4(Y_{k_l})=0$, then by Theorem 1.3 in \cite{MR3380342},
      we can see that  $(X_{k_l},Y_{k_l})_{l\in \mathbb{N}}$ are asymptotically free semicircles with variance $(\sigma_X^2,\sigma_Y^2)$. Since $\sigma_X^2+\sigma_Y^2=\sigma^2$, $\{X_{k_l}+Y_{k_l}\}_{l \in \mathbb{N}}$ converges in distribution into $S(0,\sigma^2)$, which concludes the proof.
  \end{proof}

  \section{Proof for $q$-Wigner chaos\label{q-case}}
  We start with proving the $q$-analog of Lemma \ref{keylemma}.
 
  \begin{lm}\label{q-keylemma}
      Let $0\le q \le 1$ and $ f \in H^{\otimes m},  g \in H^{\otimes n}$ be fully symmetric, and let $X=I_q(f)$ and $Y_k=I_q(g)$ be the $q$-Wigner integrals.
      We set $\| f\|_{\mathcal{F}_q(H)}=\sigma_X$ and $\| g\|_{\mathcal{F}_q(H)}=\sigma_Y$. Assume $m,n$ have opposite parities. Then, we have 
      \[\tau_q[(X+Y)^4]-(2+q^{m^2})\sigma_X^4-(2+q^{n^2})\sigma_Y^4 -(4+2q^{mn})\sigma_X^2\sigma_Y^2
      =\kappa_X+\kappa_Y +\kappa_{X,Y}
      \]
      where 
     \begin{align*}
         \kappa_X&=\tau_q(X^4)-(2+q^{m^2})\sigma_X^4\\
         \kappa_Y&=\tau_q(Y^4)-(2+q^{n^2})\sigma_Y^4\\
         \kappa_{X,Y}&=4\tau_q(X^2Y^2)+2\tau_q(XYXY)-(4+2q^{mn})\sigma_X^2\sigma_Y^2\\
         &=4\mathrm{cov}(X^2,Y^2)+2(\tau_q[XYXY]-q^{mn}\sigma_X^2\sigma_Y^2).
     \end{align*}
     Moreover, $\kappa_X,\kappa_Y,\kappa_{X,Y}\ge 0$.
  \end{lm}
\begin{proof}
   As in the proof of Lemma \ref{keylemma}, we have
   \[\tau_q[(X+Y)^4]=\tau_q(X^4)+\tau_q(Y^4)+4\tau_q(X^2Y^2)+2\tau_q(XYXY).\]
   By Proposition 3.2 in \cite{MR3089666}, we have $\kappa_X,\kappa_Y\ge 0$ when $q\ge 0$ and $ f, g$ are fully symmetric.
   By contraction formula, we have 
   \[XY=\sum_{k=0}^{m\wedge n}I_q( f\overset{k}{\frown_q} g),\]
   and 
   \[\tau_q(X^2Y^2)=\|XY\|_2^2=\sum_{k=0}^{m\wedge n}\| f\overset{k}{\frown_q} g\|_{\mathcal{F}_q(H)}^2.\]
   Similarly, we have
   \[\tau_q(XYXY)=\sum_{k=0}^{m\wedge n}\langle  f \overset{k}{\frown_q} g, g \overset{k}{\frown_q} f \rangle_q. \]
   Since $( f \overset{k}{\frown_q} g)^*= g^* \overset{k}{\frown_q} f^*= g \overset{k}{\frown_q} f$ and $\ast$ is anti-unitary operator with respect to $q$-inner product, $4\tau_q(X^2Y^2)+2\tau_q(XYXY)$ is equal to
   \[4\| f\otimes  g\|^2_{\mathcal{F}_q(H)}+2\langle  f \otimes  g, g \otimes  f\rangle_q+2\sum_{k=1}^{m\wedge n}\| f\overset{k}{\frown_q} g\|_{\mathcal{F}_q(H)}^2+2\sum_{k=1}^{m\wedge n}\| f\overset{k}{\frown_q} g+ g\overset{k}{\frown_q} f\|_{\mathcal{F}_q(H)}^2.\]
   Finally, we prove
   \[ 4\| f\otimes  g\|_{\mathcal{F}_q(H)}^2\ge 4\sigma_X^2\sigma_Y^2,\quad 2\langle  f \otimes  g, g \otimes  f\rangle_q\ge 2q^{mn}\sigma_X^2\sigma_Y^2,\]
   which implies $\kappa_{X,Y}\ge 0$. By changing $ f$ and $ g$ if necessary, we may assume $m>n$.
   Note that 
   \begin{align*}
       \| f\otimes  g\|^2_{\mathcal{F}_q(H)}& = \langle  f\otimes  g,P_{q}^{(m+n)} f \otimes  g \rangle\\
       &=\langle  f\otimes  g,R_{m,m+n}(P_q^{(m)} f \otimes P_q^{(n)} g) \rangle \\
       &=\sigma_X^2\sigma_Y^2 +\langle  f\otimes  g,(R_{m,m+n}-I_{m+n})(P_q^{(m)} f \otimes P_q^{(n)} g) \rangle.
   \end{align*}
   Since $ f$ and $ g$ are fully symmetric, we have $P_q^{(m)} f=[m]_q!  f$ and $P_q^{(n)} g=[n]_q!  g$. Thus, it suffices to check the positivity of $\langle  f\otimes  g,(R_{m,m+n}-I_{m+n})( f \otimes  g) \rangle$. Note that
   \[(R_{m,m+n}-I_{m+n})( f \otimes  g)(\vec{t})=\sum_{\substack{\pi \in \quotient{S_{m+n}}{S_m\times S_n}\\ \pi \neq 1}}q^{\mathrm{inv(\pi)}}( f \otimes  g)(\pi^{-1}(\vec{t})). \]
   For each $\pi \in S_{m+n}$, we can decompose $\{1,2,\ldots,m+n\}$ into four disjoint sets
   \begin{align*}A_1=\{ i\le m |\ \pi^{-1}(i)>m\}&,\qquad A_2=\{i\le m|\ \pi^{-1}(i)\le m\},\\ A_3=\{i>m |\ \pi^{-1}(i)>m\}&,\qquad
   A_4=\{ i>m|\ \pi^{-1}(i)\le m\}.
   \end{align*}
    We set $\vec{t}_{A_i}=(t_{i_1},\ldots,t_{i_j})$ for $A_i=\{i_1<i_2\cdots <i_j\}$ ($i=1,2,3,4$). Since $ f$ and $ g$ are fully symmetric, we have
    \begin{align*}
    ( f\otimes  g) (\vec{t})&=  f(\vec{t}_{A_1},\vec{t}_{A_2}) g(\vec{t}_{A_3},\vec{t}_{A_4})\\
    ( f\otimes  g)(\pi^{-1}(\vec{t}))&= f(\vec{t}_{A_2},\vec{t}_{A_4}) g(\vec{t}_{A_1},\vec{t}_{A_3}).
    \end{align*}
   By using symmetry of $ f$ and $ g$ again, this implies 
    \[\langle  f \otimes  g, \pi( f\otimes  g)\rangle =\| f \overset{\alpha(\pi)}{\frown} g\|^2_{L^2(\mathbb{R}_+^{m+n-2\alpha(\pi)})}. \]
    where we set $\alpha(\pi)=\#A_1=\#A_4$.
    Since we assume $q$ is non-negative, we can see
    \[\sum_{\substack{\pi \in \quotient{S_{m+n}}{S_m\times S_n}\\ \pi \neq 1}}q^{\mathrm{inv(\pi)}}\| f \overset{\alpha(\pi)}{\frown} g\|^2_{L^2(\mathbb{R}_+^{m+n-2\alpha(\pi)})} \ge 0,\]
    which implies $\| f \otimes  g\|^2_{\mathcal{F}_q(H)}\ge \sigma_X^2\sigma_X^2$. 
    
    To see $\langle  f \otimes  g, g \otimes  f\rangle_q\ge q^{mn}\sigma_X^2\sigma_Y^2$, we compute $\langle  f \otimes  g, g \otimes  f\rangle_q$ as follows
    \begin{align*}
        \langle  f \otimes  g, g \otimes  f\rangle_q&=\langle  f \otimes  g,P^{(m+n)}_q( g \otimes  f)\rangle\\
        &=\langle  f \otimes  g,R_{n,m+n}(P^{(n)}_q g \otimes P^{(m)}_q f)\rangle\\
        &=q^{mn}\sigma_X^2\sigma_Y^2+[m]_q![n]_q!\langle  f \otimes  g,(R_{n,m+n}-q^{mn}\sigma_{n,m})( g \otimes  f)\rangle,
    \end{align*}
    where $\sigma_{n,m}\in \quotient{S_{m+n}}{S_n\times S_m}$ is the flip of $H^{\otimes n}\otimes H^{\otimes m}$, namely $\sigma_{n,m}( g \otimes  f) = f \otimes  g $ for $ f \in H^{\otimes m}$. Note that $\mathrm{inv}(\sigma_{n,m})=q^{mn}$ and 
    \[(R_{n,m+n}-q^{mn}\sigma_{n,m})( g \otimes  f)(\vec{t})= \sum_{\substack{\pi \in \quotient{S_{m+n}}{S_n\times S_m}\\ \pi \neq \sigma_{n,m}}}q^{\mathrm{inv}(\pi)} ( g \otimes  f)(\pi^{-1}(\vec{t})).\]
    In this case, for each $\pi \in S_{m+n}$, we decompose $\{1,\ldots,m+n\}$ into the following six sets
    \begin{align*}
        B_1=\{i\le n|\ \pi^{-1}(i)\le n\}&, \qquad B_2=\{i\le n|\  \pi^{-1}(i)>n\},\\
        B_3=\{n< i\le m|\ \pi^{-1}(i)\le n\}&,\qquad B_4=\{n< i\le m|\  \pi^{-1}(i)>n\},\\
        B_5=\{i>m|\  \pi^{-1}(i)\le n\}&,\qquad B_6=\{ i>m|\ \pi^{-1}(i)>n\}.
    \end{align*}
    Since $ f$ and $ g$ are fully symmetric, we have
    \begin{align*}
    ( f\otimes  g)(\vec{t})&= f(\vec{t}_{B_1},\vec{t}_{B_2},\vec{t}_{B_3},\vec{t}_{B_4}) g(\vec{t}_{B_5},\vec{t}_{B_6})\\
   ( g \otimes  f)(\pi^{-1}(\vec{t})) &= g(\vec{t}_{B_1},\vec{t}_{B_3},\vec{t}_{B_5}) f(\vec{t}_{B_2},\vec{t}_{B_4},\vec{t_{B_6}}).
    \end{align*}
    Thus, we have
    \[ \langle  f \otimes  g,\pi( g \otimes  f)\rangle=\| f\overset{\beta(\pi)}{\frown} g\|^2_{L^2(\mathbb{R}_+^{m+n-2\beta(\pi)})}\]
    where we set $\beta(\pi)=\#B_1+\#B_3=\#B_6$.
    Since $q$ is non-negative, we obtain
     \[\sum_{\substack{\pi \in \quotient{S_{m+n}}{S_n\times S_m}\\ \pi \neq \sigma_{n,m}}}q^{\mathrm{inv}(\pi)}\| f\overset{\beta(\pi)}{\frown} g\|^2_{L^2(\mathbb{R}_+^{m+n-2\beta(\pi)})}\ge 0 \]
     which implies $\langle  f \otimes  g, g \otimes  f\rangle_q\ge q^{mn}\sigma_X^2\sigma_Y^2$.
\end{proof}
Before we state the fourth moment theorem for this case, we define the following distribution which corresponds with the limit distribution of our fourth moment theorem for the $q$-case. 
\begin{dfn}\label{mixedQ}
    Let $Q=(q_{ij})_{1\le i,j \le d}\in M_d(\mathbb{R})$ be a symmetric matrix with $\max|q_{ij}|\le 1$. Then the mixed $Q$-Gaussian $X_v$ with respect to $v \in \mathbb{R}^d$ defined by its moment   
    \[\tau_q[X_v^{n}]=\sum_{\alpha:[n]\to [d]}\sum_{\pi \in P_2(n)} \prod_{\substack{i<j<k<l\\ (i,k),(j,l) \in \pi}}q_{\alpha(i)\alpha(j)} \prod_{(i,j)\in \pi}v_{\alpha(i)}^2\delta_{\alpha(i)\alpha(j)}.\]
\end{dfn}
\begin{rem}
This definition comes from the mixed $Q=(q_{ij})$-relations in \cite{MR1289833},
\[c_i^*c_j-q_{ij}c_jc_i^*=\delta_{ij}.\] For $v={}^t(v_1,\ldots,v_d)$, we can see that $X_v$ has the same distribution as $\sum_{i=1}^dv_i X_i$ where $X_i=c_i+c_i^*$. It is known that, when $\max|q_{ij}|<1$, each $c_i$ is a bounded operator, which implies $X_v$ is also bounded. 
   
\end{rem}
Now, we prove the $q$-analog of \cite{basseoconnor2025fourthmoment}. Since their result is about the case of $q=1$, we consider the case of $0\le q<1$. 
  \begin{thm}
      Let $0\le q<1$ and $m,n \in \mathbb{N}$ have opposite parities. Let $\{ f_k\}_{k\in \mathbb{N}} \subset H^{\otimes m}, \{ g_k\}_{k\in \mathbb{N}} \subset H^{\otimes n}$ be fully symmetric, and let $X_k=I_q( f_k)$ and $Y_k=I_q( g_k)$ be the $q$-Wigner integrals. Assume $\|f_k\|_{\mathcal{F}_q(H)}$ and $\|g_k\|_{\mathcal{F}_q(H)}$ converge and define 
      \[ \sigma_X=\lim_{k \to \infty}\| f_k\|_{\mathcal{F}_q(H)} \quad \& \quad  \sigma_Y=\lim_{k \to \infty}\| g_k\|_{\mathcal{F}_q(H)}.\]
      Then the following are equivalent
      \begin{enumerate}
          \item $\lim_{k \to \infty}\tau_q[(X_k+Y_k)^4]=(2+q^{m^2})\sigma_X^4+(2+q^{n^2})\sigma_Y^4 +(4+2q^{mn})\sigma_X^2\sigma_Y^2$.
           \item $\displaystyle{\lim_{k\to\infty} f_k \overset{p}{\frown} f_k = 0}$ in $H^{\otimes 2(m-p)}$ for $p=1,\ldots,m-1$ and $\displaystyle{\lim_{k\to\infty} g_k \overset{p}{\frown} g_k = 0}$ in $H^{\otimes 2(n-p)}$ for $p=1,\ldots,n-1$.
          \item $X_k+Y_k $ converges in distribution toward the mixed $Q$-Gaussian with respect to $\begin{pmatrix}
              \sigma_X\\ \sigma_Y
          \end{pmatrix}$
          where 
          \[Q=\begin{pmatrix}
              q^{m^2} & q^{mn} \\
              q^{mn} & q^{n^2}
          \end{pmatrix}.\]
       \end{enumerate}
  \end{thm}
  \begin{proof}
  Note that since operator norms of all $X_k+Y_k$ are uniformly bounded on $n$ if $0\le q <1$ by Haagerup inequality \cite{MR1811255}, convergence in moments is equivalent to convergence in distribution. We can see that (3) implies (1) by computing the fourth moment of mixed $Q$-Gaussian. Actually, by following the definition \ref{mixedQ}, we have the eight possible choices of $(\alpha(1),\alpha(2),\alpha(3),\alpha(4))$; 
  \begin{align*}
      &(1,1,1,1),(2,2,2,2),(1,1,2,2),(1,2,2,1),
      \\&(2,2,1,1),
      (2,1,1,2)
      ,(1,2,1,2),(2,1,2,1).
  \end{align*}
  From $(1,1,1,1)$ and $(2,2,2,2)$, we get $(2+q^{m^2})\sigma_X^4$ and $(2+q^{n^2})\sigma_Y^4$ since we have two non-crossing pair partitions and one crossing pair partition.
  From $(1,1,2,2)$, $(1,2,2,1)$
$(2,2,1,1)$, and $(2,1,1,2)$, we have $4\sigma_X^2\sigma_Y^2$ since we have only one non-crossing pair partition which respects $\alpha$. From $(1,2,1,2)$ and $(2,1,2,1)$, we get $2q^{mn}\sigma_X^2\sigma_Y^2$ since we have only one crossing pair partition which respects $\alpha$.

If we assume (1), then Lemma \ref{q-keylemma} tells us that
  \[ \lim_{k\to \infty}\tau_q(X_k^4)=(2+q^{m^2})\sigma_X^4,\quad \lim_{k\to \infty}\tau_q(Y_k^4)=(2+q^{n^2})\sigma_Y^4,\]
  which is equivalent to (2) by Theorem 3.1 in \cite{MR3089666} (see the equivalence of the conditions (i) and (ii)).
  
  To prove (2) implies (3), it is enough to check that each moment of $X_k+Y_k$ converges to that of the mixed $Q$-Gaussian.  
  By Theorem 3.1 in \cite{MR3089666} (see the equivalence of the conditions (ii) and (iv)), the joint moments of $(X_k,Y_k)$ converge to those of the mixed $Q$-Gaussians $(X,Y)$ with respect to $\begin{pmatrix}
    \sigma_X\\0
\end{pmatrix}$ 
and $\begin{pmatrix}
   0\\ \sigma_Y
\end{pmatrix}$. Thus, each moment of $X_k+Y_k$ converges to that of $X+Y$ and this implies $X_k+Y_k$ converges in distribution to the mixed $Q$-Gaussian with respect to $\begin{pmatrix}
              \sigma_X\\ \sigma_Y
          \end{pmatrix}$. 
\end{proof}

\bibliographystyle{amsalpha}
\bibliography{reference.bib}
\end{document}